\documentclass[preprint,11pt]{elsarticle}

\usepackage[utf8]{inputenc}
\usepackage{amsfonts,amssymb,amsmath}
\usepackage{amsthm} 
\usepackage{graphicx,epsfig,verbatim,url}
\usepackage{exscale,path,hyperref}
\usepackage{verbatim,url}
\usepackage[normalem]{ulem}
\usepackage{color}

\newtheorem{theorem}{Theorem}
\newtheorem{lemma}[theorem]{Lemma}

\newtheorem{Ex}[theorem]{Example}
\newtheorem{Rem}[theorem]{Remark}

\newcommand{\N}{\mathbb{N}}
\newcommand{\Proj}{\mathbb{P}}
\newcommand{\CC}{\mathbb{C}}
\newcommand{\Ca}{{\cal C}_a^r}
\newcommand{\epsq}{{\varepsilon (Q)}}
\newcommand{\ve}{\varepsilon}
\newcommand{\etaq}{{\eta (Q)}}

\newcommand{\cP}{{\cal P}}
\newcommand{\cS}{{\cal S}}
\newcommand{\cR}{{\cal R}}
\newcommand{\cL}{{\cal L}}
\newcommand{\orbq}{{\cal O}(Q)}
\newcommand{\cloq}{\overline{\orbq}}
\newcommand{\ok}{{\cal O}({\cal K}_a)}
\newcommand{\clok}{\overline{\ok}}

\newcommand{\la}{\lambda}
\newcommand{\nr}{\mbox{\rm nrank}\,}

\newcounter{algo}[section]

\renewcommand{\leq}{\leqslant}
\renewcommand{\geq}{\geqslant}

\renewcommand{\theta}{\vartheta}

\DeclareMathOperator{\diag}{diag}

\def\BC{\mathbb C}\def\BP{\mathbb P}
\def\ra{{\mathord{\;\rightarrow\;}}}
\def\hd{, \hdots ,}\def\ep{\epsilon}
\def\ot{{\mathord{ \otimes } }}

\def\tdim{\operatorname{dim}}
\def\inv{{}^{-1}}

\begin{document}

\begin{frontmatter}

\title{An explicit description of the irreducible components of the set of matrix pencils with bounded normal rank\tnoteref{proj}}

\author[fernando]{Fernando De Ter\'an\corref{cor}}
\author[froilan]{Froil\'an M. Dopico}
\author[joseph]{J. M. Landsberg}
\address[fernando]{Departamento de Matem\'aticas, Universidad Carlos III de Madrid, Avda. Universidad 30, 28911 Legan\'es,
Spain. {\tt fteran@math.uc3m.es}.}
\address[froilan]{Departamento de Matem\'aticas, Universidad Carlos III de Madrid, Avda. Universidad 30, 28911 Legan\'es,
Spain. {\tt dopico@math.uc3m.es}.}
\address[joseph]{Department of Mathematics, Texas A\& M University, Mailstop 3368 College Station, TX 77843-3368. {\tt jml@math.tamu.edu}.}
\cortext[cor]{corresponding author.}
\tnotetext[proj]{This work was partially
supported by the Ministerio de Econom\'{i}a y Competitividad of Spain
through grants MTM-2012-32542, MTM2015-68805-REDT, and MTM2015-65798-P  (F. De Ter\'an, F. M. Dopico), and  NSF
grant DMS-1405348 (J. M. Landsberg).}

\begin{abstract} The set of $m\times n$ singular matrix pencils with normal rank at most $r$ is an algebraic set with $r+1$ irreducible components. These components are the closure of the orbits (under strict equivalence) of $r+1$ matrix pencils which are in Kronecker canonical form. In this paper, we provide a new explicit description of each of these irreducible components which is a parametrization of each component. Therefore one can explicitly construct any pencil in each of these components. The new description of each of these irreducible components consists of the sum of $r$ rank-$1$ matrix pencils, namely, a column polynomial vector of degree at most $1$ times a row polynomial vector of degree at most $1$, where we impose one of these two vectors to have degree zero. The number of row vectors with zero degree determines each irreducible component.
\end{abstract}


\end{frontmatter}

\section{Introduction}\label{sec:intro}

We are concerned in this paper with {\em singular} matrix pencils $A+\la B$, with $A,B\in\CC^{m\times n}$. This includes rectangular
pencils ($m\neq n$) and square ones ($m=n$) with $\det(A+\la B)$ identically zero as a polynomial in $\la$. More precisely,
our interest focuses on the set  $\cP_r^{m\times n}$  of $m\times n$ matrix pencils with complex coefficients and normal rank at
most $r$, with $r\leq\min\{m,n\}$ if $m\neq n$ and $r\leq n-1$ if $m=n$.

In the contexts where matrix pencils usually arise, e.g., systems of first order ordinary differential equations with constant coefficients $Ax+Bx'=f(t)$,
the relevant information is encoded in the {\em Kronecker canonical form} of the pencil (in the
following, KCF, or KCF$(A+\la B)$ when it refers to a particular pencil). This is the canonical form under strict equivalence of
matrix pencils (see Section \ref{sec:def}). The computation of the KCF of a given pencil $A+\la B$ is a delicate task, because it is
not a continuous function of the entries of $A$ and $B$ (see, e.g., \cite{dehoyos}). Nonetheless, when a good algorithm (for instance, the backward stable one in \cite{vd79}) is
used to compute the KCF, the output is the KCF of a pencil $\widetilde A+\la \widetilde B$, \lq\lq nearby\rq\rq\ to the exact one, more precisely, a KCF that contains the exact KCF in its orbit closure, as explained in the next paragraph. In this setting, the analysis of the geometry of the set of $m\times n$ matrix pencils may be useful \cite{eek1,eek2}. In particular,
the knowledge of all KCFs of the pencils included in the orbit closure of a given KCF
could improve our understanding of possible failures of the algorithms, and to develop enhanced
versions of these algorithms.

Two $m\times n$ matrix pencils $Q_1(\la)$ and $Q_2(\la)$ are said to be {\em strictly equivalent} if there exist two constant nonsingular matrices $E\in\CC^{m\times m}$ and $F\in\CC^{n\times n}$ such that $EQ_1(\la)F=Q_2(\la)$. We identify each orbit under strict equivalence with the KCF of any pencil in this orbit (by definition, they all have the same KCF). Then we say that some KCF, $ K_1+\la K_2$, degenerates to the KCF  $\widetilde K_1+\la \widetilde K_2$
if $\widetilde K_1+\la \widetilde K_2$ belongs to the closure of the orbit of $K_1+\la K_2$. In other words, if there is a sequence of matrix pencils, $A_m+\la B_m$, all having the same KCF, namely $K_1+\la K_2$, which converges to a pencil whose KCF is $\widetilde K_1+\la \widetilde K_2$. There are some cases where it is easy to determine, even at a first glance, whether a given KCF degenerates to some other one or not. This happens, for instance, with the following two pencils in KCF:
$$
K(\la)=\left(\begin{array}{cc}\la&1\\0&0\end{array}\right)\qquad\mbox{and}\qquad
\widetilde K(\la)=\left(\begin{array}{cc}\la&0\\0&0\end{array}\right).
$$
It holds that $K(\la)$ degenerates to $\widetilde K(\la)$, since the sequence $\left\{K^{(m)}(\la)\right\}_{m\in\N}\,$, with
$$
K^{(m)}(\la)=\left(\begin{array}{cc}\la&1/m\\0&0\end{array}\right),
$$
consists of pencils which are strictly equivalent to $K(\la)$ and it converges to $\widetilde K(\la)$. Note that both $K(\la)$ and $\widetilde K(\la)$ have the same normal rank, namely $1$ (we refer the reader to Section \ref{sec:def} for all notions we are using along the Introduction). However, it is not easy, in general, to know whether a given KCF degenerates to some other KCF or not.
(Although there are simple necessary conditions, e.g., the normal rank  of the first must be at least
the normal rank of the second.) Consider the following two pencils in KCF:
\begin{equation}\label{kktilde}
K(\la)=\left(\begin{array}{ccc|ccc}\la&1&&&\\&\la&1&&\\\hline&&&\la&&\\&&&1&\la&\\&&&&1&\la\\&&&&&1\end{array}\right),\quad\mbox{and}\quad
\widetilde K(\la)=\left(\begin{array}{cc|cccc}\la&1&&&\\\hline&&\la&1&\\&&&\la&1&\\&&&&\la&1\\\hline&&&&0&0\\&&&&0&0\end{array}\right).
\end{equation}
It is clear by the normal ranks that $\widetilde K(\la)$ cannot degenerate to $K(\la)$, but the question as to
whether or not  $K(\la)$ can degenerate to $\widetilde K(\la)$ is more subtle. This can be determined as explained in the following paragraph.


Necessary and sufficient conditions for the inclusion of orbit closures of any two given KCFs have been known since the 1990's \cite{bongartz,dehoyos,pokr}. These conditions enable one to determine, for example,
that $K(\la)$ in \eqref{kktilde} degenerates to $\widetilde K(\la)$ (see \cite[Th. 3.1]{eek2}).
Moreover the partial containment order of orbit closures of $m\times n$ matrix pencils is also known \cite{eek2}, and software tools are also available to get the complete Hasse diagram of the inclusion relation between orbit closures of $m\times n$ matrix pencils \cite{stratigraph}. The stratification of structured KCFs of structured matrix pencils or, more in general, of canonical eigenstructures of structured matrix polynomials, is currently an active area of research where many problems remain open \cite{dk14,dk-preprint}. In the characterization of the inclusion relation between orbit closures, the normal rank of the pencils plays a prominent role (see \cite[Th. 3.1]{eek2}), so it makes sense to have a closer look at the set of matrix pencils with bounded normal rank.

The set of $m\times n$ singular matrix pencils is an algebraic set, so it is natural to analyze it from the point of view of algebraic geometry. The approach to the description of subsets of matrix pencils using algebraic geometry can be traced back to the 1980's with the work by Waterhouse \cite{waterhouse}, who identified the irreducible components of $\cP^{n\times n}_{n-1}$. More recently, the $r+1$ irreducible components of $\cP_{r}^{m\times n}$ have been described in \cite{dd-orbits}. These components are given as the  orbit closures
of certain KCFs, which are termed the \lq\lq generic\rq\rq\  KCFs of $m\times n$ matrix pencils with normal rank at most $r$. This name emphasizes the fact that any $m\times n$ KCF with normal rank at most $r$ is in the closure of at least one orbit among the $r+1$ orbits corresponding to the generic KCFs. This description is motivated by possible numerical applications, since it deals with nearby canonical structures. However, given a matrix pencil which is not in KCF, it is not easy, in general, to determine whether or not it belongs to a certain component using this description. Even if the pencil is given in KCF, to determine whether the pencil belongs to some irreducible component requires one to check certain majorization conditions \cite[Th. 3.1]{eek2}.

Recently, a new description of $\cP_r^{m\times n}$ was presented in \cite{dd-addendum} as the union of $r+1$ subsets, in order to solve open low-rank perturbation problems \cite{ddm-weier}. The germ of this description was already present in \cite{dd-kron} for pencils with normal rank exactly $r$, but it was not used again until \cite{dd-addendum}. It seems natural to ask whether these $r+1$ subsets are related with the $r+1$ irreducible components provided in \cite{dd-orbits}.

In this paper we prove that the subsets mentioned in the preceeding paragraph coincide with the $r+1$ irreducible components of $\cP_r^{m\times n}$. This provides a description of the irreducible components of $\cP_r^{m\times n}$ that makes no use of the KCF. The description is given in terms of a decomposition of an $m\times n$ matrix pencil with normal rank at most $r$ as a sum of $r$ pencils $u(\la)v(\la)^T$ with rank at most $1$ and having an specific  $0/1$ degree pattern of the columns $u(\la)$ and rows $v(\la)^T$ of each summand $u(\la)v(\la)^T$. The new description is based on the decomposition of $\cP_r^{m\times n}$ in \cite{dd-addendum}, as the union of $r+1$ different subsets that correspond to each of these $0/1$ degree patterns. We show that each set in this decomposition coincides with exactly one orbit closure in the previous description of $\cP_r^{m\times n}$. We provide two different proofs of this fact. The first one makes use of tools and techniques from linear algebra and
matrix analysis, whereas the second one follows an algebraic geometry approach.

The paper is organized as follows. In Section \ref{sec:def} we introduce the basic notions, tools, and notation used throughout the paper, we recall the previous results on the description of $\cP_r^{m\times n}$ mentioned above and, finally, we state our main result (Theorem \ref{thm:main}). In
Section \ref{sec:main} we present the first  proof of Theorem \ref{thm:main}, based on a linear algebra approach,
together with several auxiliary technical results.
Section \ref{sec:ag} is devoted to the second proof of Theorem  \ref{thm:main}, that uses tools from algebraic geometry.
Although the second proof is considerably shorter, it requires familiarity with basic concepts of algebraic geometry. By contrast, the first one can be followed by anyone with an elementary background in matrix pencils. Finally, in Section \ref{sec:conclusion} we summarize the contributions of the paper.


\section{Notation, definitions, previous results, and statement of the main result}\label{sec:def}

Throughout the paper, $I_k$ denotes the $k\times k$ identity matrix. By $\CC(\la)$ and $\CC[\la]$ we denote, respectively, the field of rational functions and the ring of polynomials in the variable $\la$ with complex coefficients. We also denote by $\CC[\la]^n$ the set of column vectors with $n$ coordinates in $\CC[\la]$. Vectors in $\CC[\la]^n$ are termed {\em vector polynomials}. Analogously, $\CC[\la]^{m\times n}$ and $\CC(\la)^{m\times n}$ denote, respectively, the set of $m\times n$ matrix polynomials and the set of $m\times n$ rational matrices. The {\em degree} of a vector polynomial $v$, denoted by $\deg v$, is the maximum degree of its components. Instead of $A+\la B$ we will use, in general, the shorter notation $Q(\la)$ for a matrix pencil.

The {\em normal rank} of a matrix pencil $Q(\la)$, denoted by $\nr Q$, is the rank of $Q(\la)$ considered as a matrix over $\CC(\la)$. In other words, $\nr Q$ is the size of the largest non-identically zero minor of $Q(\la)$ \cite{eek2} (see also \cite[Ch. XII, \S3]{gant}, where the name {\em rank} is used instead). For brevity, a matrix pencils with normal rank at most $1$ will be termed a {\em rank-$1$ pencil}.

 Given a matrix pencil $Q(\la)$, the {\em orbit under strict equivalence of $Q(\la)$}, denoted by $\orbq$, is the set of matrix pencils which are strictly equivalent to $Q(\la)$. By $\cloq$ we denote the closure of $\orbq$ in the standard topology of $\CC^{2mn}$, after identifying $\CC^{2mn}$ with the set of $m\times n$ matrix pencils with complex entries. It is known \cite{hinrichsen} that this coincides with the closure of $\orbq$ in the Zariski topology in $\CC^{2mn}$. This result is a very special case of a classical result in algebraic geometry \cite[Thm 2.33 p. 38]{mumford} that the Zariski and classical closures of a Zariski open subset of an irreducible projective variety coincide.

Let us recall, for the sake of completeness, the KCF of a matrix pencil $Q(\la)$ \cite[Ch. XII]{gant}.

\begin{theorem}\label{thm:kcf} {\bf(Kronecker canonical form)}
Each complex matrix pencil $Q(\la)$ is strictly equivalent to a direct sum of blocks of the following types:

\begin{itemize}

\item[{\bf(1)}]{\bf Right singular blocks} (of {\rm order} $\ve$): \[L_{\ve}={\left(\begin{array}{ccccc}\la & 1 &  &  &    \\ & \la & 1 &  &   \\ &  & \ddots & \ddots &    \\ &  &  & \la & 1 \end{array}\right)}_{\ve\times(\ve+1)}\,.\]

\item[{\bf(2)}]  {\bf Left singular blocks} (of {\rm order} $\eta$): $L_{\eta}^{ T},$ where $L_{\eta}$ is a right singular block.

\item[{\bf(3)}] {\bf Finite blocks:} $J_{ k}(\mu)+\la I_k$, where $J_{ k}(\mu)$ is a Jordan block of size $k\times k$ associated with $\mu\in\CC$, that is,
\[J_{ k}(\mu)=\left(\begin{array}{ccccc}\mu & 1 &   &   &   \\  & \mu & 1 &   &   \\  &   & \ddots & \ddots &   \\  &   &   & \mu & 1 \\  &   &   &   & \mu \end{array}\right)_{k\times k}\,.\]

\item[{\bf(4)}] {\bf Infinite blocks:} $N_{u}=I_{ u}+\la J_{ u}(0)$.
\end{itemize}

\noindent This direct sum of blocks is uniquely determined, up to permutation of blocks, and is known as the {\rm Kronecker canonical form} of $Q(\la)$.
\end{theorem}


Note that KCF$(Q)$ may contain singular blocks of the form $L_0$ or $L_0^T$. The first one adds one null column to the KCF and no rows, whereas the second one adds one null row and no columns.

Some known facts about KCF$(Q)$ will be used throughout the paper. We refer the reader to \cite[Ch. XII]{gant} for more information on this topic. In the first place, if $Q(\la)$ is $m\times n$ and $\nr Q=r$, then the number of left and right singular blocks in KCF$(Q)$ is $m-r$ and $n-r$, respectively.

Left (respectively, right) singular blocks in KCF$(Q)$ are associated with vectors in the {\em left} (resp., {\em right}) rational {\em nullspace of $Q(\la)$}
\begin{eqnarray*}
  {\cal N}_{\ell}(Q)\!\! &:=& \!\left\{y(\la)^T \in\CC(\la)^{1 \times m}
                           \,:\,y(\la)^T Q(\la)\equiv0^T\right\},\\
                            (\mbox{resp.\ } {\cal N}_r(Q)\!\! &:=& \!\left\{x(\la)\in\CC(\la)^{n \times 1}
                           \,:\, Q(\la)x(\la)\equiv0\right\}).
\end{eqnarray*}
More precisely, if $\varepsilon_1\leq\cdots\leq\varepsilon_p$ and $\eta_1\leq\cdots\leq\eta_q$ are the orders of the right and left singular blocks in KCF$(Q)$, respectively, then there are bases $\{x_1(\la),\hdots,x_p(\la)\}$ and $\{y_1(\la)^T,\hdots,y_q(\la)^T\}$ of $ {\cal N}_{r}(Q)$ and $ {\cal N}_{\ell}(Q)$, respectively, formed by vector polynomials with $\deg x_i=\varepsilon_i$, for $i=1,\hdots,p$, and $\deg y_j=\eta_j$, for $j=1,\hdots,q$. The numbers $\varepsilon_1\leq\cdots\leq\varepsilon_p$ and $\eta_1\leq\cdots\leq\eta_q$ are known as, respectively, the {\em column} and {\em row minimal indices of $Q(\la)$} or, also, as the {\em right} and {\em left minimal indices of $Q(\la)$}. In these conditions, we denote:
$$
\varepsilon(Q):=\sum_{i=1}^p\varepsilon_i,\qquad\mbox{and}\qquad
\eta(Q):=\sum_{i=1}^q\eta_i\,
$$
for the sum of right and left minimal indices of $Q(\la)$, respectively.

We follow the notation from \cite{dd-orbits,eek1,eek2}. In particular, the notation $r_i(Q)$ and $\ell_i(Q)$ is used, respectively, for the number of right and left singular blocks in KCF$(Q)$ of order at least $i$, for $i=0,1,\hdots$ Then, we define
$$
\begin{array}{c}
\cR(Q):=(r_0(Q),r_1(Q),r_2(Q),\hdots)\\
\cL(Q):=(\ell_0(Q),\ell_1(Q),\ell_2(Q),\hdots).
\end{array}
$$
These lists, together with the list ${\cal J}_{\mu}(Q)$ of Weyr characteristics of $Q(\la)$ for the eigenvalue $\mu$ (see \cite[p. 680]{eek2}), are key in describing the necessary and sufficient conditions for inclusion of orbit closures under strict equivalence of two given matrix pencils. The majorization of lists, $(a_1,a_2,\hdots)\geq(b_1,b_2,\hdots)$, is understood as $\sum_{i=1}^ja_i\geq\sum_{i=1}^jb_i$, for all $j=1,2,\hdots$ (see \cite[p. 671]{eek2}). Also, the sum of the list $(a_1,a_2,\hdots)$ and the number $s$ is the list obtained by adding $s$ to every element in the list, that is, $(a_1,a_2,\hdots)+s:=(a_1+s,a_2+s,\hdots)$.

\begin{theorem}\label{thm:majorization}{\rm (\cite[Th. 3.1]{eek2})} Given two $m\times n$ matrix pencils $P(\la)$ and $Q(\la)$, then $\overline{{\cal O}(Q)}\subseteq\overline{{\cal O}(P)}$ if and only if the following three conditions hold:
\begin{itemize}
\item[{\rm(i)}] $\cR(P)+\nr P\geq\cR(Q)+\nr Q$,
\item[{\rm(ii)}] $\cL(P)+\nr P\geq\cL(Q)+\nr Q$,
\item[{\rm(iii)}] ${\cal J}_{\mu}(P)+r_0(P)\leq{\cal J}_{\mu}(Q)+r_0(Q)$, for any $\mu\in\CC\cup\{\infty\}$,
\end{itemize}
The third inequality is equivalent to ${\cal J}_{\mu}(P)+\ell_0(P)\leq{\cal J}_{\mu}(Q)+\ell_0(Q)$, since for any $m\times n$ matrix pencil $M(\la)$, it holds that $r_0(M)-\ell_0(M)=n-m$.
\end{theorem}

The set $\cP_r^{m\times n}$ is an algebraic subset of $\CC^{2mn}$, since it is  the whole $\CC^{2mn}$ if $r=\min\{m,n\}$, or it is defined as the common zeros of a set of polynomials in $2mn$ variables if $r\leq\min\{m,n\}-1$. More precisely, these polynomials are all the $(r+1)\times(r+1)$ minors of an arbitrary $m\times n$ matrix pencil. We are interested in describing the irreducible components of $\cP_r^{m\times n}$. There is a known description of these components as the orbit closures of certain KCF's. For the sake of completeness, we reproduce this result here.

\begin{theorem}\label{thm:orbits}{\rm (\cite[Th. 3.5]{dd-orbits})}
Let $r$ be an integer with $1\leq r\leq \min\{m,n\}$ if $m\neq n$ and $1\leq r\leq n-1$ if $m=n$. Then the set $\cP_r^{m\times n}$ is a closed set which has exactly $r+1$ irreducible components in the Zariski topology. These irreducible components are $\overline \ok$, for $a=0,1,\hdots,r$, where
\begin{equation}\label{oka}
{\cal K}_a(\la):=\diag (\underbrace{L_{\alpha+1},\hdots,L_{\alpha+1}}_{s},\underbrace{L_\alpha,\hdots,L_\alpha}_{n-r-s},\underbrace{L_{\beta+1}^T,\hdots,L_{\beta+1}^T}_{t},\underbrace{L_\beta^T,\hdots,L_\beta^T}_{m-r-t}),
\end{equation}
with $a=\alpha(n-r)+s$ and $r-a=\beta(m-r)+t$ being the Euclidean divisions of $a$ and $r-a$ by, respectively, $n-r$ and $m-r$.
\end{theorem}

The description of the irreducible components of $\cP_r^{m\times n}$ given in Theorem \ref{thm:orbits} extends the one by Waterhouse in \cite[Th. 1]{waterhouse} (see also \cite[Cor. 2]{de}), valid only for the irreducible components of the set of $n\times n$ singular matrix pencils (namely, $\cP_{n-1}^{n\times n}$). Later on, Demmel and Edelman provided the generic KCFs of the set of singular $m\times n$ matrix pencils \cite[Cor. 1]{de}, which coincide with the KCFs ${\cal K}_a(\la)$ described in Theorem \ref{thm:orbits} for the cases $r=\min\{m,n\}$ if $m\neq n$, and $r=n-1$ if $m=n$. However, the connection with the irreducible components is not considered in \cite{de}. The original statement of \cite[Th. 3.5]{dd-orbits} we have reproduced in Theorem \ref{thm:orbits} does not include the case $r=\min\{m,n\}$ when $m\neq n$, though the proof is also valid for this case, and for this reason we include it here.

The following description of $\cP_r^{m\times n}$ was recently presented in \cite{dd-addendum} for square matrix pencils. However, it is also valid for rectangular ones, and we state it for this more general case, but we omit the proof since it is completely analogous to that in \cite{dd-addendum}.

\begin{lemma}\label{lemma:tech} {\rm (\cite[Lemma 3.1]{dd-addendum})} Let $r\leq \min\{m,n\}$ be an integer. For each $a=0,1,\hdots,r$, define
$$
{\cal C}_a^r:=\left\{u_1(\la)v_1(\la)^T+\cdots+u_r(\la)v_r(\la)^T\,:\ \ \begin{array}{l}
u_i(\la)\in\CC[\la]^m,v_j(\la)\in\CC[\la]^n,\\
\deg u_i\leq 1, \mbox{\rm for $i=1,\hdots,r$,}\\
\deg v_j\leq 1, \mbox{\rm for $j=1,\hdots,r$,}\\
\deg u_1=\cdots=\deg u_a = 0,\\
\deg v_{a+1}=\cdots=\deg v_r=0
\end{array}\right\}.
$$
Then:
\begin{equation}\label{decomp}
\cP_r^{m\times n}={\cal C}^r_0\cup{\cal C}^r_1\cup\cdots\cup{\cal C}^r_r\,.
\end{equation}
\end{lemma}

Both Theorem \ref{thm:orbits} and Lemma \ref{lemma:tech} provide a description of $\cP_r^{m\times n}$ as the union of $r+1$ sets.
It has been recently proved in \cite[Prop. 5.1 ]{dd-addendum} that both descriptions coincide in the case $r=1$, namely, that $\overline{{\cal O}({\cal K}_0)}={\cal C}^1_0$ and $\overline{{\cal O}({\cal K}_1)}={\cal C}^1_1$. It is natural to ask whether the same holds for arbitrary $r$, namely, whether the $r+1$ sets in Theorem \ref{thm:orbits} coincide with the $r+1$ sets in Lemma \ref{lemma:tech} (after an appropriate reordering if needed). To answer this question is our main purpose. More precisely, the main goal of this paper is to prove the following result:

\begin{theorem}\label{thm:main}
Let $\overline\ok$ and $\Ca$, for $a=0,1,\hdots,r$, be the sets defined in Theorem {\rm\ref{thm:orbits}} and Lemma {\rm\ref{lemma:tech}}, respectively. Then:
\begin{itemize}
\item[{\rm (a)}] $\overline\ok=\Ca$.
\item[{\rm (b)}] The set $\cP_r^{m\times n}$ is closed in the Zariski topology of $\CC^{2mn}$ and has exactly $r+1$ irreducible components. These irreducible components are $\Ca$, for $a=0,1,\hdots,r$.
\end{itemize}
\end{theorem}

Claim (b) in Theorem \ref{thm:main} is an immediate consequence of claim (a) and Theorem \ref{thm:orbits}. So
it remains to prove claim (a),  and this is the goal of the first proof we offer of Theorem \ref{thm:main}. In contrast, the second proof we present, via algebraic geometry, allows us to directly obtain (b) without using (a). More precisely, the second proof proceeds by exhibiting (the projectivization of) $\Ca$ as the image of a regular map from a product of projective spaces, which immediately implies it is Zariski closed and irreducible. This, together with Lemma \ref{lemma:tech}, proves (b). Part (a) then follows from the fact that ${\cal K}_a(\la)\in\Ca$, which implies that $\overline{{\cal O}({\cal K}_a)}\subseteq\Ca$ by the invariance of $\Ca$ under strict equivalence and the fact that $\Ca$ is closed, together with a dimensional count.

Theorem \ref{thm:main} provides a new description of the irreducible components of $\cP_r^{m\times n}$. We present, in Sections \ref{sec:la} and \ref{sec:ag},  the two different proofs of Theorem \ref{thm:main} mentioned above. The first one, in Section \ref{sec:la}, is based on a purely linear algebra approach, whereas the second one, in Section \ref{sec:ag}, uses standard facts from algebraic geometry.
The main difference is the first proof uses the classical topology (where closure is
defined by taking limits), so one must study limits, whereas the second proof uses the
Zariski topology (where the closed sets are,  by definition, the zero sets of polynomials), which, combined with basic
facts about projective varieties, leads to a quick proof.


\section{The linear algebra approach}\label{sec:la}\label{sec:main}

The expression for matrix pencils in ${\Ca}$ given in Lemma \ref{lemma:tech} is closely related with the KCF. This connection is underlying in a relevant portion of the first proof of Theorem \ref{thm:main}, and it is explained in Remark \ref{rem-R} for further reference.

\begin{Rem}\label{rem-R}
A given $m\times n$ matrix pencil $Q(\la)$ in {\rm KCF} with $\nr Q=r$ can be expressed in a natural way as in the definition of $\Ca$ in Lemma {\rm\ref{lemma:tech}} as follows. Let
$$
Q(\la)=\diag (L_{\varepsilon_1},\hdots,L_{\varepsilon_p},L_{\eta_1}^T,\hdots,L_{\eta_q}^T,J_Q),
$$
where $J_Q$ is a direct sum of Jordan blocks (that is, of types {\rm(3)} and {\rm(4)} in Theorem {\rm\ref{thm:kcf}}). Let $J_Q$ have size $s\times s$. Then
$$
\begin{array}{ccc}
m&=&\varepsilon_1+\cdots+\varepsilon_p+\eta_1+\cdots+\eta_q+q+s=\varepsilon(Q)+\eta(Q)+q+s,\\
n&=&\varepsilon_1+\cdots+\varepsilon_p+p+\eta_1+\cdots+\eta_q+s=\varepsilon(Q)+\eta(Q)+p+s,
\end{array}
$$
and $r=m-q=n-p$, so $s=r-\varepsilon(Q)-\eta(Q)$. Then, following the proof of Lemma {\rm3.1} in {\rm\cite{dd-addendum}}, we can write
\begin{equation}\label{qdecomp}
\begin{array}{cl}
Q(\la)=&u_1(\la)v_1(\la)^T+\cdots+u_{\varepsilon(Q)}(\la)v_{\varepsilon(Q)}(\la)^T\\
&+\widetilde u_1(\la)\widetilde v_1(\la)^T+\cdots+
\widetilde u_{\eta(Q)}(\la)\widetilde v_{\eta(Q)}(\la)^T\\
&+\widehat u_1(\la)\widehat v_1(\la)^T+\cdots+\widehat u_{s}(\la)\widehat v_{s}(\la)^T\,,
\end{array}
\end{equation}
where
\begin{itemize}
\item[(a)] $\deg u_1=\cdots=\deg u_{\varepsilon(Q)}=0$,
\item[(b)] $\deg \widetilde v_1=\cdots=\deg \widetilde v_{\eta(Q)}=0$,
\item[(c)] for each $i=1,\hdots,s$, we can choose either $\deg \widehat u_i=0$ or $\deg \widehat v_i=0$.
\end{itemize}
Moreover:
\begin{itemize}
\item[(a)] The sum $u_1(\la)v_1(\la)^T+\cdots+u_{\varepsilon(Q)}(\la)v_{\varepsilon(Q)}(\la)^T$ corresponds to the right singular blocks, $L_{\varepsilon_1},\hdots,L_{\varepsilon_p}$.
\item[(b)] The sum $\widetilde u_1(\la)\widetilde v_1(\la)^T+\cdots+\widetilde u_{\eta(Q)}(\la)\widetilde v_{\eta(Q)}(\la)^T$ corresponds to the left singular blocks, $L_{\eta_1}^T,\hdots,L_{\eta_q}^T$.
\item[(c)] The sum $\widehat u_1(\la)\widehat v_1(\la)^T+\cdots+\widehat u_{s}(\la)\widehat v_{s}(\la)^T$ corresponds to the regular part $J_Q$.
\end{itemize}
More precisely, each right singular block $L_{\varepsilon_i}$, with $\varepsilon_i>0$, can be decomposed as a sum of $\varepsilon_i$ rank-$1$ pencils of the form $u(\la)v(\la)^T$, with $\deg u=0$ and $\deg v=1$, as indicated in the proof of {\rm\cite[Lemma 3.1]{dd-addendum}}. Adding up the sums corresponding to all right singular blocks with positive order we get the sum in (a) above. Each left singular block $L_{\eta_j}^T$, with $\eta_j>0$, can be written as a sum of $\eta_j$ rank-$1$ pencils of the form $\widetilde u(\la)\widetilde v(\la)^T$ with $\deg \widetilde v=0$, as indicated in the proof of {\rm\cite[Lemma 3.1]{dd-addendum}}. Adding up the sums corresponding to all left singular blocks with positive order we get the sum in (b) above. Finally, any Jordan block of size $k\times k$ (finite or infinite) can be written as a sum of $k$ rank-$1$ pencils of the form $\widehat u(\la)\widehat v(\la)^T$, with either $\deg \widehat u=0$ or $\deg\widehat v=0$ and $\deg\widehat u=1$. This is shown in the proof of {\rm\cite[Lemma 3.1]{dd-addendum}} for either all rows $\widehat v(\la)^T$ with degree $0$ or all columns $\widehat u(\la)$ with degree $0$. To get the general decomposition, having $i$ rows with degree $0$ and $k-i$ columns with degree $0$, for $0\leq i\leq k$, we can decompose any $k\times k$ Jordan block, denoted by $J$, as:
$$
J=e_1{\rm Row}_1(J)+\cdots+e_{i-1}{\rm Row}_{i}(J)+J_{ii}e_ie_i^T+{\rm Col}_{i+1}(J)e_{i+1}^T+\cdots+{\rm Col}_k(J)e_k^T,
$$
with $e_j$ being the $j$th column of $I_k$. Adding upthe sums corresponding to all Jordan blocks in $Q(\la)$ we arrive to the
sum in (c) above. The decomposition \eqref{qdecomp} will be often used in the proof of the main result.
\end{Rem}

\begin{Rem}\label{rem-ka}
An immediate consequence of the decomposition explained in Remark {\rm\ref{rem-R}} is that the pencil ${\cal K}_a(\la)$ in Theorem {\rm\ref{thm:orbits}} belongs to the set ${\cal C}_a^r$ in Lemma {\rm\ref{lemma:tech}}. To see this, just note that $\varepsilon({\cal K}_a)=a$, $\eta({\cal K}_a)=r-a$, and that ${\cal K}_a(\la)$ has no regular part..
\end{Rem}

In order to give our first proof of Theorem \ref{thm:main}, we first state and prove several auxiliary results that we will use along the proof. The proof of Lemma \ref{lemma:sum} is omitted, since it is a standard fact.

\begin{lemma}\label{lemma:epseta}
If $Q(\la)\in\Ca$ and $\nr Q=r$, then
\begin{itemize}
\item[{\rm(i)}] $\epsq\leq a$, and
\item[{\rm (ii)}] $\etaq\leq r-a$.
\end{itemize}
\end{lemma}
\begin{proof} Since $Q(\la)\in\Ca$, it has a decomposition like in the statement of Lemma \ref{lemma:tech}, with at most $a$ row vectors $v_1(\la)^T,\hdots,v_a(\la)^T$ having degree exactly $1$. Then $\epsq\leq a$, by the last sentence in the statement of Lemma 2.8 in \cite{dd-kron}. Part (ii) is an immediate consequence of part (i) and the facts:
\begin{itemize}
\item If $Q(\la)\in\Ca$ then $Q(\la)^T\in{\cal C}_{r-a}^r$ (with size $n\times m$), and
\item $\eta(Q)=\varepsilon(Q^T)$.
\end{itemize}
\end{proof}

\begin{lemma}\label{lemma:sum}
Let $\cS=(\beta_1,\hdots,\beta_s)$ be a list of nonnegative integers, and let $r_i(\cS)$ be the number of elements in $\cS$ which are greater than or equal to $i$, for $i=1,2,\hdots$. Then
$$
\sum_{i=1}^\infty r_i(\cS)=\beta_1+\cdots+\beta_s.
$$
\end{lemma}

Note that, as a consequence of Lemma \ref{lemma:sum}, if $Q(\la)$ is any matrix pencil, then
\begin{eqnarray}
\displaystyle\sum_{i=1}^\infty r_i(Q)=\epsq,
\label{sumesp}\\
\displaystyle\sum_{i=1}^\infty \ell_i(Q)=\etaq.\label{sumeta}
\end{eqnarray}

\begin{lemma}\label{lemma:inclusion}
If $Q(\la)$ is an $m\times n$ matrix pencil such that
\begin{itemize}
\item[{\rm (i)}] $\nr Q=r$,
\item[{\rm(ii)}] $\epsq\leq a$, and
\item[{\rm (iii)}] $\etaq\leq r-a$,
\end{itemize}
then $\cloq\subseteq \overline\ok$.
\end{lemma}
\begin{proof} Since $\nr Q=r=\nr {\cal K}_a$, we have $r_0(Q)=r_0({\cal K}_a)$. Moreover, since KCF$({\cal K}_a)$ has no Jordan blocks at all (neither finite nor infinite), looking at the majorization conditions for $\cloq\subseteq \overline\ok$ in Theorem \ref{thm:majorization}, it suffices to prove that
\begin{itemize}
\item[(a)] $(r_1({\cal K}_a),r_2({\cal K}_a),\hdots)\geq (r_1(Q),r_2(Q),\hdots)$, and
\item[(b)] $(\ell_1({\cal K}_a),\ell_2({\cal K}_a),\hdots)\geq (\ell_1(Q),\ell_2(Q),\hdots)$.
\end{itemize}

To prove (a) and (b) first note that
$$
\begin{array}{c}
(r_1({\cal K}_a),r_2({\cal K}_a),\hdots)=(\overbrace{n-r,\hdots,n-r}^\alpha,s,0,0,\hdots)\\
(\ell_1({\cal K}_a),\ell_2({\cal K}_a),\hdots)=(\underbrace{m-r,\hdots,m-r}_\beta,t,0,0,\hdots),
\end{array}
$$
with $\alpha,\beta,s,t$ being as in Theorem \ref{thm:orbits}. Since, for all $i=1,2,\hdots$ the inequalities
$$
r_i(Q)\leq n-r,\qquad\mbox{and}\qquad \ell_i(Q)\leq m-r
$$
hold, it follows that
$$
\begin{array}{cc}
\displaystyle\sum_{i=1}^kr_i(Q)\leq\sum_{i=1}^kr_i({\cal K}_a)=k(n-r),&\mbox{for $1\leq k\leq\alpha$,}\\
\displaystyle\sum_{i=1}^k\ell_i(Q)\leq\sum_{i=1}^k\ell_i({\cal K}_a)=k(m-r),&\mbox{for $1\leq k\leq\beta$}.
\end{array}
$$
Now, if there is some $k\geq\alpha+1$ such that
$$
\sum_{i=1}^kr_i(Q)>\sum_{i=1}^kr_i({\cal K}_a)=\alpha(n-r)+s=a,
$$
or, if there is some $k\geq\beta+1$ such that
$$
\sum_{i=1}^k\ell_i(Q)>\sum_{i=1}^k\ell_i({\cal K}_a)=\beta(m-r)+t=r-a,
$$
then by \eqref{sumesp} or \eqref{sumeta}, respectively, it should be
$$
\epsq\geq\sum_{i=1}^kr_i(Q)>a,
$$
or
$$
\etaq\geq\sum_{i=1}^k\ell_i(Q)>r-a,
$$
which is in contradiction with hypothesis (ii) or (iii), respectively.
\end{proof}

In the following, we make use of the Frobenius norm. Let us recall that, for any complex matrix $M=(m_{ij})$, the {\em Frobenius norm} of $M$ is $\| M\|_F:=\left(\sum_{i,j}|m_{ij}|^2\right)^{1/2}$. In particular, for a vector $u=\left[\begin{array}{ccc}u_1&\hdots&u_n\end{array}\right]^T\in\CC^n$, the Frobenius norm of $u$ is the standard $2$-norm $\|u\|_2:=\left(\sum_{i=1}^n|u_i|^2\right)^{1/2}$. For a complex matrix pencil $A+\la B$ the Frobenius norm is defined as $\|A+\la B\|_F:=\|\left[\begin{array}{cc}A&B\end{array}\right]\|_F$ (the Frobenius norm for matrix pencils will be used only in the first part of the proof of Theorem \ref{thm:main}).

The following lemma is a direct consequence of the fact that the set of linearly dependent $r$-tuples of vectors in $\CC^n$ is of measure zero
in the set of all $r$-tuples of vectors in $\CC^n$.

\begin{lemma}\label{lemma:pert}
Let $w_1,\hdots,w_r\in\CC^n$, with $r\leq n$, and $\epsilon>0$. Then there exist $w_1^\epsilon,\hdots, w_r^\epsilon\in\CC^n$ such that $\{ w_1^\epsilon,\hdots, w_r^\epsilon\}$ is a linearly independent set and $\|w_i- w_i^\epsilon\|_2\leq\epsilon$, for $i=1,\hdots,n$.
\end{lemma}

As noted in Remark \ref{rem-R}, any right singular block $L_k$ can be written as the sum of $k$ rank-$1$ pencils of the form $u_1(\la)v_1(\la)^T+\cdots+u_k(\la)v_k(\la)^T$, with $\deg u_1=\cdots=\deg u_k=0$ and $\deg v_1=\cdots=\deg v_k=1$. However, in the proof of Theorem \ref{thm:main} we need to write $L_k$ as a sum of rank-$1$ pencils $u(\la)v(\la)^T$ with some of the rows $v(\la)^T$ having degree zero instead. The following result shows that this can be done at a cost of using $k+1$ summands instead of $k$, and that we can set as many rows $v(\la)^T$ with degree zero as we want (up to $k+1$).

\begin{lemma}\label{lemma:L}
For each $j=0,1,\hdots,k+1$ we can decompose a right singular block $L_k$ as a sum of $k+1$ rank-$1$ vector polynomials with degree at most $1$
\begin{equation}\label{lk-decomp}
L_k=u_1(\la)v_1(\la)^T+\cdots+u_{k+1}(\la)v_{k+1}(\la)^T,
\end{equation}
where $u_i(\la)\in\CC[\la]^k,v_i(\la)\in\CC[\la]^{k+1}$, for $i=1,\hdots,k+1$, and $\deg u_1=\cdots=\deg u_j=\deg v_{j+1}=\cdots=\deg v_{k+1}=0$.
\end{lemma}
\begin{proof} A decomposition as in the statement is not necessarily unique. We provide one such decomposition by considering the following four cases. Along the proof $e_i^{(k)}$ denotes the $i$th column of the $k\times k$ identity matrix.
\begin{itemize}
\item Case 1: $j=0$. Set $u_i(\la)={\rm Col}_i\,L_k$ and $v_i(\la)=e_i^{(k+1)}$, for $i=1,\hdots,k+1$.
\item Case 2: $j=k+1$. This is the case described in Remark \ref{rem-R}, where all column vectors $u_i(\la)$ have degree zero, and just $k$ nonzero summands are needed.
\item Case 3: $j=k$. Set
\begin{itemize}
\item $u_i(\la)=e_i^{(k)}$ and $v_i(\la)^T={\rm Row}_i L_k$, for $i=1,\hdots,k-1$,
\item $u_k(\la)=e_k^{(k)}$ and $v_k(\la)^T=\left(e_{k+1}^{(k+1)}\right)^T$,
\item $u_{k+1}(\la)=\la e_k^{(k)}$ and $v_{k+1}(\la)^T=\left(e_k^{(k+1)}\right)^T$.
\end{itemize}
\item Case 4: $1\leq j\leq k-1$. Set
\begin{itemize}
\item $u_i(\la)=e_i^{(k)}$ and $v_i(\la)^T={\rm Row}_i L_k$, for $i=1,\hdots,j$,
\item $u_{j+1}(\la)=\la e_{j+1}^{(k)}$ and $v_{j+1}(\la)^T=\left(e_{j+1}^{(k+1)}\right)^T$,
\item $u_{i}(\la)={\rm Col}_i\,L_k$ and $v_{i}(\la)^T=\left(e_i^{(k+1)}\right)^T$, for $i=j+2,\hdots,k+1$.
\end{itemize}
\end{itemize}
\end{proof}

The following result combines $m\times n$ and $n\times m$ matrix pencils by means of transposition. To avoid confusion, we introduce the notation ${\cal K}_a^{m\times n}$ to explicitly indicate the size of the matrix pencil ${\cal K}_a$ in Theorem \ref{thm:orbits}. The proof is straightforward from the majorization conditions in Theorem \ref{thm:majorization} and we omit it.

\begin{lemma}\label{lemma:transpose}
Let $Q(\la)$ be an $m\times n$ pencil with $\nr Q\leq r$. If $Q(\la)\in\overline{{\cal O}({\cal K}_{a}^{m\times n})}\subseteq \cP_r^{m\times n}$, then $Q(\la)^T\in\overline{{\cal O}({\cal K}_{r-a}^{n\times m})}\subseteq\cP_r^{n\times m}$.
\end{lemma}

\noindent{\bf First proof of Theorem \ref{thm:main}.} Let us first prove that $\Ca\subseteq\clok$. Note that Lemmas \ref{lemma:epseta} and \ref{lemma:inclusion} together imply that if $Q(\la)\in\Ca$ and $\nr Q=r$, then $Q(\la)\in\clok$. It remains to prove the inclusion for matrix pencils in $\Ca$ having normal rank smaller than $r$. So let $Q(\la)\in\Ca$ with $\nr Q<r$. Since $Q(\la)\in\Ca$, it can be written as
$$
Q(\la)=u_1(\la)v_1(\la)^T+\cdots+u_r(\la)v_r(\la)^T,
$$
with $\deg u_1=\hdots=\deg u_a=\deg v_{a+1}=\hdots=\deg v_r=0$. Then we can write
\begin{equation}\label{u}
u_i(\la)=\left\{\begin{array}{cc}u_{i0}\,,&1\leq i\leq a,\\u_{i0}+\la u_{i1}\,,&a+1\leq i\leq r\end{array}\right.,
\end{equation}
and
\begin{equation}\label{v}
  v_i(\la)=\left\{\begin{array}{cc}v_{i0}+\la v_{i1}\,,&1\leq i\leq a,\\v_{i0}\,,&a+1\leq i\leq r\end{array}\right..
\end{equation}
By Lemma \ref{lemma:pert}, for each $\epsilon>0$, there are $u_{10}^\epsilon,\hdots,u_{r0}^\epsilon\in\CC^m$, and $v_{10}^\epsilon,\hdots,v_{r0}^\epsilon\in\CC^n$ such that $\{u_{10}^\epsilon,\hdots,u_{r0}^\epsilon\}$ and $\{v_{10}^\epsilon,\hdots,v_{r0}^\epsilon\}$ are linearly independent and $\|u_{i0}-u_{i0}^\epsilon\|_2\leq\epsilon$, $\|v_{i0}-v_{i0}^\epsilon\|_2\leq\epsilon$, for $i=1,\hdots,r$. Set:
$$
u_i^\epsilon(\la)=\left\{\begin{array}{cc}u_{i0}^\epsilon\,,&1\leq i\leq a,\\u_{i0}^\epsilon+\la u_{i1}\,,&a+1\leq i\leq r\end{array}\right.
$$
and
$$
 v_i^\epsilon(\la)=\left\{\begin{array}{cc}v_{i0}^\epsilon+\la v_{i1}\,,&1\leq i\leq a,\\v_{i0}^\epsilon\,,&a+1\leq i\leq r\end{array}\right..
$$
Now we are going to see that:
\begin{itemize}
\item[(a)] Both $\{u_1^\epsilon(\la),\hdots,u_r^\epsilon(\la)\}$ and $\{v_1^\epsilon(\la),\hdots,v_r^\epsilon(\la)\}$ are linearly independent sets over $\CC(\la)$,

\item[(b)] $Q_\epsilon(\la):=u_1^\epsilon(\la)v_1^\epsilon(\la)^T+\cdots+u_r^\epsilon(\la)v_r^\epsilon(\la)^T$ has normal rank exactly $r$,

\item[(c)] $Q_\epsilon(\la)\in\Ca$, and

\item[(d)] $\|Q(\la)-Q_\epsilon(\la)\|_F\leq r\epsilon^2+\epsilon\, \alpha(Q)$,
\end{itemize}
where $\alpha(Q)$ is a quantity depending on $Q$, and does not depend on $\epsilon$.

Claim (a) is an immediate consequence of Lemma 2.6 in \cite{dd-kron}. For claim (b), just notice that $Q_\epsilon(\la)$ is the product:
$$
Q_\epsilon(\la)=\left[\begin{array}{ccc}u_1^\epsilon(\la)&\cdots&u_r^\epsilon(\la)\end{array}\right]\left[\begin{array}{ccc}v_1^\epsilon(\la)&\cdots&v_r^\epsilon(\la)\end{array}\right]^T
=U_\epsilon(\la)V_\epsilon(\la)^T,
$$
with $r\leq\min\{m,n\}$, and where both $U_\epsilon(\la)$ and $V_\epsilon(\la)$ have full column normal rank, by (a). Then the product $U_\epsilon(\la)V_\epsilon(\la)^T$ has full normal rank as well. Claim (c) is an immediate consequence of the definition of $\Ca$. To prove claim (d) we first note that
\begin{equation}\label{uv}
\begin{array}{lc}
u_i^\epsilon(\la)v_i^\epsilon(\la)^T-u_i(\la)v_i(\la)^T\\=\left\{
\begin{array}{cc}
\begin{array}{c}(u_{i0}^\epsilon-u_{i0})(v_{i0}^\epsilon-v_{i0})^T+u_{i0}(v_{i0}^\epsilon-v_{i0})^T+\\
(u_{i0}^\epsilon-u_{i0})v_{i0}^T+\la(u_{i0}^\epsilon-u_{i0})v_{i1}^T,
\end{array}&\mbox{$1\leq i\leq a$}\\
\begin{array}{c}(u_{i0}^\epsilon-u_{i0})(v_{i0}^\epsilon-v_{i0})^T+u_{i0}(v_{i0}^\epsilon-v_{i0})^T+\\
(u_{i0}^\epsilon-u_{i0})v_{i0}^T+\la u_{i1}(v_{i0}^\epsilon-v_{i0})^T,
\end{array}&\mbox{$a+1\leq i\leq r$.}
\end{array}
\right.
\end{array}
\end{equation}
Therefore,
$$
\begin{array}{cl}
\|Q(\la)-Q_\epsilon(\la)\|_F &\leq\displaystyle\sum_{i=1}^{a}\|u_i^\epsilon(\la)v_i^\epsilon(\la)^T-u_i(\la)v_i(\la)^T\|_F\\
&+\displaystyle\sum_{i=a+1}^r\|u_i^\epsilon(\la)v_i^\epsilon(\la)^T-u_i(\la)v_i(\la)^T\|_F\\
&\leq\displaystyle\sum_{i=1}^{a}(\epsilon^2+\epsilon\left(\|u_{i0}\|_2+\|v_{i0}\|_2+\|v_{i1}\|_2\right)\\
&+\displaystyle\sum_{i=a+1}^r(\epsilon^2+\epsilon\left(\|u_{i0}\|_2+\|v_{i0}\|_2+\|u_{i1}\|_2\right))\\
&=r\epsilon^2+\epsilon\displaystyle \left(\sum_{i=1}^a \|u_{i0}\|_2+\|v_{i0}\|_2+\|v_{i1}\|_2\right.\\
&+\displaystyle\left.\sum_{i=a+1}^r \|u_{i0}\|_2+\|v_{i0}\|_2+\|u_{i1}\|_2\right)=r\epsilon^2+\epsilon\, \alpha(Q),
\end{array}
$$
where the last inequality follows from \eqref{uv} and the basic inequality $\|\left[\begin{array}{cc}A&B\end{array}\right]\|_F\leq\|A\|_F+\|B\|_F$.

Now, from (b) and (c), and the result for pencils in $\Ca$ having normal rank exactly $r$, it follows that $Q_\epsilon(\la)\in\clok$. But, by (d), we have $\displaystyle\lim_{\epsilon\rightarrow0}Q_\epsilon(\la)=Q(\la)$ and, since $\clok$ is closed, we conclude that $Q(\la)\in\clok$, as wanted.

Now, we are going to prove the converse inclusion, namely that $\clok\subseteq\Ca$. So let $Q(\la)\in\clok$ with $\nr Q=\widetilde r\leq r$. We consider separately the following three cases.
\begin{itemize}
\item[(C1)] $\varepsilon(Q)=a$. In this case, and following \cite[Lemma 2.8]{dd-kron}, we can write
$$
Q(\la)=u_1(\la)v_1(\la)^T+\cdots+
u_{\widetilde r}(\la)v_{\widetilde r}(\la)^T,
$$
with $\deg u_1=\cdots=\deg u_a=\deg v_{a+1}=\cdots=\deg v_{\widetilde r}=0$, which shows that $Q(\la)\in\Ca$ (note that, if $\widetilde r<r$, we can add $r-\widetilde r$ summands with $u_{\widetilde r+1}(\la)\equiv\cdots\equiv u_r(\la)\equiv0$ and $v_{\widetilde r+1}(\la),\hdots,v_r(\la)$ being arbitrary constant nonzero vectors).

\item[(C2)] $\varepsilon(Q)>a$. In this case, it must be $\widetilde r<r$. To see this, note that $\widetilde r=r$ implies, by (i) in Theorem \ref{thm:majorization}, that $\cR({\cal K}_a)\geq\cR(Q)$, which in turn implies, by \eqref{sumesp}, that $\varepsilon(Q)\leq a$.

Since $\Ca$ is closed under strict equivalence, we may assume $Q(\la)$ given in KCF and, following Remark \ref{rem-R}, we can write:
\begin{equation}\label{q}
\begin{array}{cl}
Q(\la)=&u_1(\la)v_1(\la)^T+\cdots+u_a(\la)v_a(\la)^T\\
&+u_{a+1}(\la)v_{a+1}(\la)^T+\cdots+u_{\varepsilon(Q)}(\la)v_{\varepsilon(Q)}(\la)^T\\
&+u_{\varepsilon(Q)+1}(\la)v_{\varepsilon(Q)+1}(\la)^T+\cdots+u_{\widetilde r}(\la)v_{\widetilde r}(\la)^T,
\end{array}
\end{equation}
with $\deg u_1=\cdots=\deg u_{\varepsilon(Q)}=\deg v_{\varepsilon(Q)+1}=\cdots=\deg v_{\widetilde r}=0$. As in Remark \ref{rem-R}, the first $\varepsilon(Q)$ summands in the right hand side of \eqref{q} correspond to the right singular blocks, and, assuming the right singular blocks of $Q(\la)$ ordered in nondecreasing order, the sum $u_{a+1}(\la)v_{a+1}(\la)^T+\cdots+u_{\varepsilon(Q)}(\la)v_{\varepsilon(Q)}(\la)^T$ corresponds to right singular blocks with largest size. Let $\alpha_1\leq\hdots\leq \alpha_{n-\widetilde r}$ be the orders of the right singular blocks of $Q(\la)$ and let $\alpha$ be as in the statement of Theorem \ref{thm:orbits}. We distinguish the following two cases:
\begin{itemize}
\item[(C2.1)] $\alpha_{n-r+1}\geq \alpha+1$. In this case, by the majorization conditions for the inclusion of orbit closures in Theorem \ref{thm:majorization} we have
\begin{equation}\label{diff1}
\sum_{i=1}^{\alpha_{n-r+1}}r_i(Q)-\sum_{i=1}^{\alpha_{n-r+1}}r_i({\cal K}_a)\leq(r-\widetilde r)\alpha_{n-r+1}.
\end{equation}
 Note that we have removed the term $r_0(Q)+\widetilde r-(r_0({\cal K}_a)+r)$ appearing in the majorization condition, since this term is zero. This is because the sum of the normal rank of an $m\times n$ matrix pencil plus its number of right singular blocks is equal to $n$.

 By \eqref{sumesp} applied to both $Q(\la)$ and ${\cal K}_a(\la)$ we get that
 \begin{equation}\label{diff2}
 \hspace{-1.4cm}\sum_{i=1}^{\alpha_{n-r+1}}r_i(Q)-\sum_{i=1}^{\alpha_{n-r+1}}r_i({\cal K}_a)=\varepsilon(Q)-a-(r_{\alpha_{n-r+1}+1}(Q)+\cdots+r_{\alpha_{n-\widetilde r}}(Q)).
 \end{equation}
 Note that, given a list $\cS=(\beta_1,\hdots,\beta_s)$ of nonnegative integers and $0\leq\beta\leq\min\cS$, for each $i\geq \beta$, the identity $r_i(\cS)=r_{i-\beta}(\cS-\beta)$ holds, where $\cS-\beta:=\{\beta_1-\beta,\hdots,\beta_s-\beta\}$. Now, let us write
 $$
\hspace{-.8cm} \alpha_{n-r+1}+\cdots+\alpha_{n-\widetilde r}=(r-\widetilde r)\alpha_{n-r+1}+(\alpha_{n-r+2}-\alpha_{n-r+1})+\cdots+(\alpha_{n-\widetilde r}-\alpha_{n-r+1}).
 $$
The previous observation and Lemma \ref{lemma:sum} lead to
$$
\begin{array}{l}
(\alpha_{n-r+2}-\alpha_{n-r+1})+\cdots+(\alpha_{n-\widetilde r}-\alpha_{n-r+1})\\
=r_1(Q-\alpha_{n-r+1})+\cdots+r_{\alpha_{n-\widetilde r}-\alpha_{n-r+1}}(Q-\alpha_{n-r+1})
\\=r_{\alpha_{n-r+1}+1}(Q)+\cdots+r_{\alpha_{n-\widetilde r}}(Q),
\end{array}
$$
so the last two equations give
 \begin{equation}\label{diff3}
\hspace{-.8cm} (r-\widetilde r)\alpha_{n-r+1}+r_{\alpha_{n-r+1}+1}(Q)+\cdots+r_{\alpha_{n-\widetilde r}}(Q)=\alpha_{n-r+1}+\cdots+\alpha_{n-\widetilde r}\,.
 \end{equation}
Combining equations \eqref{diff1}--\eqref{diff3} we obtain
 \begin{equation}\label{bound}
 \alpha_{n-r+1}+\cdots+\alpha_{n-\widetilde r}\geq\varepsilon(Q)-a\,.
 \end{equation}
 Equation \eqref{bound} means that the largest $r-\widetilde r$ right singular blocks of $Q(\la)$ fill at least $\varepsilon(Q)-a$ rows in $Q(\la)$. In other words, as described in Remark \ref{rem-R}, in the $\varepsilon(Q)-a$ rows corresponding to the sum $u_{a+1}(\la)v_{a+1}(\la)^T+\cdots+u_{\varepsilon(Q)}(\la)v_{\varepsilon(Q)}(\la)^T$ in \eqref{q}, there are no more than $r-\widetilde r$ right singular blocks involved. As a consequence, equation \eqref{q} can be decomposed as
$$
\begin{array}{cl}
Q(\la)=&u_1(\la)v_1(\la)^T+\cdots+u_t(\la)v_t(\la)^T\\
&+u_{t+1}(\la)v_{t+1}(\la)^T+\cdots+u_a(\la)v_a(\la)^T+\cdots+u_{\varepsilon(Q)}(\la)v_{\varepsilon(Q)}(\la)^T\\
&+u_{\varepsilon(Q)+1}(\la)v_{\varepsilon(Q)+1}(\la)^T+\cdots+u_{\widetilde r}(\la)v_{\widetilde r}(\la)^T,
\end{array}
$$
where the summands in the second line correspond exactly to the $r-\widetilde r$ largest right singular blocks of $Q(\la)$, as explained in Remark \ref{rem-R}. Now, using Lemma \ref{lemma:L}, we can write the sum in the second line of the equation above as
 \begin{equation}\label{replacement}
 \begin{array}{c}
 u_{t+1}(\la)v_{t+1}(\la)^T+\cdots+u_a(\la)v_a(\la)^T+\cdots+u_{\varepsilon(Q)}(\la)v_{\varepsilon(Q)}(\la)^T\\
= \widetilde u_{t+1}(\la)\widetilde v_{t+1}(\la)^T+\cdots+\widetilde u_a(\la)\widetilde v_a(\la)^T+\cdots+\widetilde u_{\varepsilon(Q)}(\la)\widetilde v_{\varepsilon(Q)}(\la)^T\\
+ \widehat u_1(\la)\widehat v_1(\la)^T+\cdots+\widehat u_{r-\widetilde r}(\la)\widehat v_{r-\widetilde r}(\la)^T,
 \end{array}
 \end{equation}
 with $\deg \widetilde v_{a+1}=\cdots=\deg\widetilde v_{\varepsilon(Q)}=\deg\widehat v_1=\cdots=\deg\widehat v_{r-\widetilde r}=0$. Replacing this expression into \eqref{q} we arrive to an expression like the one in the definition of $\Ca$ in Lemma \ref{lemma:tech}, so $Q(\la)\in\Ca$.

\item[(C2.2)] $\alpha_{n-r+1}<\alpha+1$. In this case,
$$
\alpha_1+\cdots+\alpha_{n-r}\leq(n-r)\alpha_{n-r+1}\leq(n-r)\alpha\leq a.
$$
Hence, there are at least $n-r$ different right singular blocks in the first $a$ rows of $Q(\la)$. Since the total number of right singular blocks in $Q(\la)$ is $n-\widetilde r$, there cannot be more than $r-\widetilde r$ right singular blocks involved in the following $\varepsilon(Q)-a$ rows. Again, we can write \eqref{replacement} and replace this sum into \eqref{q} to conclude that $Q(\la)\in\Ca$.

\end{itemize}

\item[(C3)] $\varepsilon(Q)<a$. We assume $Q(\la)$ being in KCF, as in case (C2), and we consider separately the following cases:
\begin{itemize}
\item[(C3.1)] $\eta(Q)\leq r-a$. In this case, there is a decomposition of the form \eqref{qdecomp} for $Q(\la)$, where $\deg u_1=\cdots=\deg u_{\varepsilon(Q)}=\deg\widetilde v_1=\cdots=\deg\widetilde v_{\eta(Q)}=0$ and $s=\widetilde r-\varepsilon(Q)-\eta(Q)$. If $\widetilde r<r$, we can also set $\widehat u_{\widetilde r-\varepsilon(Q)-\eta(Q)+1}(\la)\equiv\cdots\equiv\widehat u_{r-\varepsilon(Q)-\eta(Q)}(\la)\equiv0$ and $\widehat v_{\widetilde r-\varepsilon(Q)-\eta(Q)+1}(\la)\equiv\cdots\equiv\widehat v_{r-\varepsilon(Q)-\eta(Q)}(\la)\equiv0$, in order to have $r$ summands in \eqref{qdecomp} instead of $\widetilde r$. Moreover, as mentioned in Remark \ref{rem-R}, claim (c), we can choose $\widehat u_i(\la)$ and $\widehat v_i(\la)$ with either $\deg \widehat u_i=0$ or $\deg\widehat v_i=0$, for each $i=1,\hdots,r-\varepsilon(Q)-\eta(Q)$. Then, since $r-\varepsilon(Q)-\eta(Q)=(a-\varepsilon(Q))+(r-a-\eta(Q))$, we can chose $a-\varepsilon(Q)$ vectors $\widehat u_i(\la)$ with degree zero (for instance, $\deg \widehat u_1=\cdots=\deg \widehat u_{a-\varepsilon(Q)}=0$) and $r-a-\eta(Q)$ vectors $\widehat v_j(\la)$ with degree zero (after the previous choice it would be $\deg \widehat v_{a-\varepsilon(Q)+1}=\cdots=\deg\widehat v_{r-\varepsilon(Q)-\eta(Q)}=0$). This gives a decomposition of $Q(\la)$ in $\Ca$.

\item[(C3.2)] $\eta(Q)>r-a$. This case can be reduced to (C2) by considering $Q(\la)^T$ instead of $Q(\la)$. To be precise, we have:
\begin{itemize}
\item[(i)] $\varepsilon(Q^T)=\eta(Q)>r-a$.
\item[(ii)] Since $Q(\la)\in\overline{{\cal O}({\cal K}_{a}^{m\times n}})$, then $Q(\la)^T\in\overline{{\cal O}({\cal K}_{r-a}^{n\times m}})$, by Lemma \ref{lemma:transpose}.
\end{itemize}
Then, (i) and (ii), together with case (C2) imply that $Q(\la)^T\in{\cal C}^r_{r-a}\subseteq{\cP}_{r-a}^{n\times m}$, and this in turn implies that $Q(\la)\in\Ca\subseteq{\cP}_r^{m\times n}$.
\end{itemize}
\end{itemize}
\hfill$\square$

Since ${\cal K}_a(\la)=\diag\left({\cal K}_a^{(1)},{\cal K}_a^{(2)}\right)$, with ${\cal K}_a^{(1)}$ having $a$ rows and ${\cal K}_a^{(2)}$ having $r-a$ columns, it is natural to wonder whether any pencil $Q(\la)\in\clok$ is strictly equivalent to a pencil of the form $\diag(Q_1,Q_2)$, with $Q_1$ having $a$ rows and $Q_2$ having $r-a$ columns. The following example shows that this is not true. This example also illustrates the construction in Lemma \ref{lemma:L}.

\begin{Ex}
Let us consider the $6\times6$ pencils $K(\la)$ and $\widetilde K(\la)$ in \eqref{kktilde}. Note that $K(\la)={\cal K}_2(\la)$ in Theorem {\rm\ref{thm:orbits}} if we set $r=5$. As mentioned in Section {\rm\ref{sec:intro}}, $\widetilde K(\la)\in\overline{{\cal O}(K)}=\overline{{\cal O}({\cal K}_2)}$, as can be easily checked by Theorem {\rm\ref{thm:majorization}}. Using the decomposition shown in the proof of Lemma {\rm \ref{lemma:L}}, this can also be seen by writing:
$$
\begin{array}{ccl}
\widetilde K(\la)=\diag(L_1,L_3)&=&e_1^{(6)}\left(\begin{array}{cccccc}\la&1&0&0&0&0\end{array}\right)\\
&&+e_2^{(6)}\left(\begin{array}{cccccc}0&0&\la&1&0&0\end{array}\right)\\
&&+\la e_3^{(6)}(e_4^{(6)})^T\\
&&+\left(\begin{array}{c}0\\0\\1\\\la\\0\\0\end{array}\right)(e_5^{(6)})^T\\
&&+ e_4^{(6)}(e_6^{(6)})^T,
\end{array}
$$
which shows that $\widetilde K(\la)\in{\cal C}_2^5=\overline{{\cal O}({\cal K}_2)}$. We note that $\widetilde K(\la)$ can not be written as $\widetilde K(\la)=\diag(\widetilde K^{(1)},\widetilde K^{(2)})$, with $\widetilde K^{(1)}$ having $a=2$ rows and $\widetilde K^{(2)}$ having $r-a=3$ columns. If such a decomposition exists, then the right singular blocks in KCF$(Q)$ would be the union of the right singular blocks of KCF$(\widetilde K_1)$ and KCF$(\widetilde K_2)$, so there would not be an $L_3$ block in KCF$(\widetilde K)$. Note, however, that $\widetilde K(\la)\not\in{\cal C}_2^4$ (by Lemma {\rm\ref{lemma:epseta} (i)}), despite $\nr \widetilde K=4$.
\end{Ex}

\section{The proof of Theorem \ref{thm:main} via algebraic geometry}\label{sec:ag}

The linear algebra proof proceeded by first showing that those pencils in $\Ca$ with normal rank exactly $r$ belong to $\overline{{\cal O}({\cal K}_a)}$, then
that all pencils of $\Ca$ are in $\clok$,
and finally that $\clok\subseteq\Ca$. The last two assertions were cumbersome to prove because they involved checking several
cases and one needed to argue with limits.

As mentioned above, in our situation one obtains the same closure via taking limits as taking the {\it Zariski closure}:
the Zariski closure of a set $X\subset \BC^N$    is the common
zero set of the space of all polynomials on $\BC^N$
that vanish on all points of $X$. (In general, the Zariski closure always contains the closure
obtained by taking limits.)

In algebraic geometry, it is often convenient to work in projective space
$\BC\BP^N$ which is the set of all lines through the origin in $\BC^{N+1}$ or equivalently
$(\BC^{N+1}\backslash 0)/\sim$ where $v\sim w$ if $v=\lambda w$ for some $\lambda\in \BC\backslash 0$.
Let $\pi : \BC^{N+1}\backslash 0\ra \BC\BP^N$ denote the projection map.
This is especially convenient when the sets of interest are invariant under rescaling,
as will be  our case. A {\em projective variety} $X\subset \BC\BP^N$ is the image under  $\pi$  of the common zero set
of a collection of  homogeneous polynomials on $\BC^{N+1}$. In particular,  a projective variety
is Zariski closed by definition. It is {\it irreducible} if it cannot be nontrivially written as the union
of two projective varieties. A subset $X\subseteq\BC\BP^N$ is Zariski closed and irreducible if and only if  $\pi^{-1}(X)\cup 0\subseteq\BC^{N+1}$ is Zariski closed and irreducible.

The following proof of Theorem \ref{thm:main} avoids the above-mentioned  difficulties by first exhibiting $\Ca$ as the image of a map whose image
is Zariski closed and invariant under multiplication by the groups of invertible $n\times n$ and $m\times m$ matrices, respectfully denoted
$GL_n$ and $GL_m$. (In the language of algebraic geometry, $\Ca$ is exhibited as a $(GL_n\times GL_m)$-variety.) Then, since
we have already seen that ${\cal K}_a(\la)$ belongs to $\Ca$ (see Remark \ref{rem-ka}), its orbit closure must belong as well. Finally,  a simple upper bound on the dimension of $\Ca$ and the
observation that $\Ca$ is irreducible, shows they coincide.

Write $(\CC^m)^p$ to denote the cartesian product of $\CC^m$ with itself $p$ times, and $V\ot W$ denotes the tensor product of the vector spaces $V$ and $W$.
Let $a\in \{ 0,1\hd r\}$.
 It will be convenient to use double indices to denote elements of
$\BC^m$ and $\BC^n$: we write $u_{\mu, \ep}\in \BC^m$ and $v_{\mu,\ep}\in \BC^n$, where $\ep\in \{0,1\}$.
Define
$$
f_a:(\BC^m)^{2r-a}\times (\BC^n)^{r+a}  \ra \BC^2\ot \BC^m\ot \BC^n
$$
by
\begin{align*}
&(u_{1, 0}\hd u_{r, 0},u_{a+1,1}\hd u_{r, 1})\times
(v_{1, 0}\hd v_{r,0},v_{1,1}\hd v_{a,1})
 \mapsto \\
&
e_1\otimes[u_{1,0}\otimes v_{1,0}+\cdots+u_{r,0}\otimes v_{r,0}]\\
&
+e_2\otimes[u_{1,0}\otimes v_{1,1}+\cdots+u_{a,0}\otimes v_{a,1}
 +u_{a+1,1}\otimes v_{a+1,0}+\cdots+u_{r,1}\otimes v_{r,0}].
\end{align*}

Recall that $\BC^m\ot \BC^n$ may be identified with the space of $m\times n$ matrices. Define a map
$$
mat: \BC^2\ot \BC^m\ot \BC^n  \ra \BC[\lambda]^{ m\times n}
$$
by sending $e_1\mapsto 1$ and $e_2\mapsto \lambda$.
Then $mat$ applied to the image of $f_a$ is exactly $\Ca$.

The proof that $\Ca$ is Zariski closed and irreducible follows completely standard arguments. For the convenience of
the reader we present them here.
Write $\bold u= (u_{1, 0}\hd u_{r, 0},u_{a+1,1}\hd u_{r, 1})$ and $\bold v= (v_{1, 0}\hd v_{r,0},v_{1,1}\hd v_{a,1})$.
Note that $f_a(\lambda \bold u, \mu\bold v)=\lambda\mu f_a(\bold u,\bold v)$, for $\lambda,\mu\in \BC\backslash 0$, so $f_a$  descends to a map
$$
 \begin{array}{cccc}
pf_a: \BC\BP^{m(2r-a)-1}\times\BC\BP^{n(r+a)-1}&\longrightarrow&\Proj(\CC^2\otimes\CC^m\otimes\CC^n) .
 \end{array}
$$
In coordinates, $f_a$ and $pf_a$ are given by the same homogeneous quadratic polynomials.
(To see this, let $a_{\sigma}$ be a basis of $\BC^m$ and $b_{\tau}$ a basis of $\BC^n$.
Write $u_{i,\ep}=\sum_\sigma u_{i,\ep,\sigma}a_{\sigma}$ and similarly for $v_{j,\ep}$.
Then the coefficient of, e.g., $e_1\ot a_{\sigma}\ot b_{\tau}$ in the image is
$\sum_{i,j=1}^r u_{i,0,\sigma}v_{j,0,\tau}$.)
More precisely, the polynomials are
linear on each projective space. In particular, the map $pf_a$
is a {\it regular map}. (A regular map from a product of projective spaces $\BP^A\times \BP^B$ to a projective space is
one defined by polynomials that are homogeneous on each   space in the product, and such that the only common zeros of these polynomials in $\BC^{A+1}\times \BC^{B+1}$ are of the form
$(0,y)$ or $(x,0)$, where $x\in \BC^{A+1}$ and $y\in  \BC^{B+1}$.) The product of projective spaces
is an irreducible projective variety (see, e.g., \cite[\S I.5.1]{shafarevich}).

Now we use two standard facts: If $X$ is an irreducible projective variety and $f: X\ra \BC\BP^N$ is a regular
map, then the image is irreducible and closed.
To see the first, note that if $f(X)=Y_1\cup Y_2$, with $Y_j$ varieties, then $X=f\inv(Y_1)\cup f\inv(Y_2)$, a contradiction.
That the image is closed is more difficult to prove, see, e.g. \cite[\S I.5.2, Thm. 2]{shafarevich}.

The above remarks prove that ${\rm Im}(pf_a)$ is Zariski closed and irreducible. Since ${\rm Im}(f_a)=\pi^{-1}({\rm Im}(pf_a))\cup0$, this implies that ${\rm Im}(f_a)$ is Zariski closed and irreducible, which in turn implies that  the set $\Ca$ is Zariski closed and irreducible. This, together with \eqref{decomp} proves that $\Ca$, for $a=0,1,\hdots,r$, are the irreducible components of $\cP_r^{m\times n}$, which is part (b) of Theorem \ref{thm:main}.

Part (a) follows from part (b), together with Theorem \ref{thm:orbits} and the uniqueness of the irreducible components. However, we can give an alternative proof, without using Theorem \ref{thm:orbits}, as follows. As mentioned above, ${\cal K}_a\in \Ca$ implies that ${\cal O}({\cal K}_a)\subseteq\Ca$ and, since $\Ca$ is Zariski closed, this in turn implies $\clok\subseteq\Ca$. The fact that $\Ca$ is
irreducible and of dimension at most
$\tdim  \clok$, will then show $\clok=\Ca$.

It remains to prove the dimension estimate. This can be done directly by computing the rank of the differential
of $f_a$ at a general point, but can easily be seen by the following argument:

Using \eqref{u} and \eqref{v} we can write any pencil $Q(\la)\in\Ca$ as
$$
Q(\la)=\sum_{i=0}^{r}u_{i,0}\ot v_{i,0}+\lambda\left(\sum_{j=0}^au_{j,0}\ot v_{j,1}+\sum_{k=a+1}^ru_{k,1}\ot v_{k,0}\right).
$$
The trailing coefficient is an arbitrary $m\times n$ matrix with rank at most $r$. The set of $m\times n$ matrices with rank at most $r$ is an algebraic set of dimension $r(m+n-r)$. The leading coefficient introduces $an+(r-a)m$ new parameters. As a consequence, the dimension of $\Ca$ is  at most  the sum of these two quantities, namely $r(2m+n-r)+a(n-m)$.
But by \cite[Th. 3.3]{dd-orbits}, $\tdim \clok=r(2m+n-r)+a(n-m)$, and since $\clok\subseteq \Ca$, equality must hold.

The proof is complete.
\hfill$\square$

\section{Conclusions}\label{sec:conclusion}

We have presented a new description of the irreducible components of the set of $m\times n$ matrix pencils with normal rank at most $r$, which covers all situations where matrix pencils are singular, namely $r\leq\min\{m,n\}$ if $m\neq n$, and $r\leq n-1$ if $m=n$. This new description is constructible in the sense that it depends on a finite number of parameters which are combined to get a sum of $r$ rank-$1$ pencils $u(\la)v(\la)^T$, in such a way that one of $u(\la)$ or $v(\la)$ has degree zero. Unlike the previously known description of these irreducible components, this new one does not require the knowledge of the Kronecker canonical form in order to determine whether a given $m\times n$ pencil of normal rank at most $r$ belongs to a certain component or not.




\begin{thebibliography}{99}

\bibitem{bongartz} K. Bongartz.
\newblock On degenerations and extensions of finite dimensional modules.
\newblock {\em Adv. Math.}, 121 (1996) 245--287.

\bibitem{dehoyos} I. De Hoyos.
\newblock Points of continuity of the Kronecker canonical form.
\newblock{\em SIAM J. Matrix Anal. Appl.}, 11 (1990) 278--300.

\bibitem{de} J. W. Demmel and A. Edelman.
\newblock The dimension of matrices (matrix pencils) with given Jordan (Kronecker) canonical form.
\newblock {\em Linear Algebra Appl.}, 230 (1995) 61--87.

\bibitem{dd-kron} F. De Ter\' an and F. M. Dopico.
\newblock Low rank perturbation of Kronecker structures without full rank.
\newblock {\em SIAM J. Matrix Anal. Appl}., 29 (2007) 496--529.

\bibitem{dd-orbits} F. De Ter\' an and F. M. Dopico.
\newblock A note on generic Kronecker orbits of matrix pencils with fixed rank.
\newblock {\em SIAM J. Matrix Anal. Appl}., 30 (2008) 491--496.

\bibitem{dd-addendum} F. De Ter\' an and F. M. Dopico.
\newblock Generic change of the partial multiplicities of regular matrix pencils under low-rank perturbations.
\newblock To appear in {\em SIAM J. Matrix Anal. Appl}.

\bibitem{ddm-weier} F. De Ter\'an, F. M. Dopico, and J. Moro.
\newblock Low rank perturbation of Weierstrass structure.
\newblock {\em SIAM J. Matrix Anal. Appl.}, 30 (2008) 538--547.

\bibitem{dk14} A. Dmytryshyn and B. K\aa gstr\"{o}m.
\newblock Orbit closure hierarchies of skew-symmetric matrix pencils.
\newblock {\em SIAM J. Matrix Anal. Appl.,} 35 (2014) 1429-1443.

\bibitem{dk-preprint} A. Dmytryshyn.
\newblock {\em Structure preserving stratification of skew-symmetric matrix polynomials}.
\newblock Technical report UMINF 15.16, Department of Computing Science, Ume\aa\ University, Sweden, 2015

\bibitem{eek1}  A. Edelman, E. Elmroth, and B. K\aa gstr\"{o}m.
\newblock A geometric approach to perturbation theory
of matrices and matrix pencils. Part I: Versal deformations.
\newblock {\em SIAM J. Matrix Anal. Appl.}, 18 (1997) 653--692.

\bibitem{eek2} A. Edelman, E. Elmroth, and B. K\aa gstr\"{o}m.
\newblock A geometric approach to perturbation theory
of matrices and matrix pencils. Part II: A stratification-enhanced staircase algorithm.
\newblock {\em SIAM J. Matrix Anal. Appl}., 20 (1999) 667--699.

\bibitem{gant} F.\ R.\ Gantmacher.
\newblock {\em The Theory of Matrices}.
\newblock Chelsea, New York, 1959.

\bibitem{hinrichsen} D. Hinrichsen and J. O'Halloran.
\newblock Orbit closures of singular matrix pencils.
\newblock{\em J. Pure Appl. Algebra}, 81 (1992) 117--137.

\bibitem{stratigraph} P. Johansson.
\newblock {\em Matrix Canonical Structure Toolbox.}
\newblock Technical report UMINF 06.15, Department of Computing Science, Ume\aa\ University, Sweden, 2006.

\bibitem{mumford} D. Mumford.
\newblock {\em Algebraic Geometry (I)}.
\newblock Springer-Verlag, Berlin, 1995.

\bibitem{pokr} A. Pokrzywa.
\newblock On perturbations and the equivalence orbit of a matrix pencil.
\newblock{\em Linear Algebra Appl.}, 82 (1986) 99--121.

\bibitem{shafarevich} I. R. Shafarevich.
\newblock {\em Basic algebraic geometry.} 1, 2nd. ed.
\newblock Springer-Verlag, Berlin, 1994.

\bibitem{vd79} P. Van Dooren.
\newblock The Computation of Kronecker's canonical form of a singular pencil.
\newblock {\em Linear Algebra Appl.}, 27 (1979) 103--140.

\bibitem{waterhouse} W. C. Waterhouse.
\newblock The codimension of singular matrix pairs.
\newblock{\em Linear Algebra Appl.}, 57 (1984) 227--245.

\end{thebibliography}
\end{document}